\documentclass{amsart}

\usepackage[T2A]{fontenc}
\usepackage[english]{babel}

\theoremstyle{plain}

\newtheorem{thm}{Theorem}
\newtheorem{pr}{Proposition}

\newtheorem{cor}{Corollary}

\theoremstyle{definition}
\newtheorem{df}{Definition}

\newcommand{\bd}{B_n}
\newcommand{\sph}{S}
\newcommand{\meal}{m}
\newcommand{\hol}{\mathcal{H}ol}
\newcommand{\Dbb}{\mathbb{D}}
\newcommand{\Tbb}{\mathbb{T}}
\newcommand{\Nbb}{\mathbb{N}}

\newcommand{\dom}{\mathcal{D}}
\newcommand{\ddo}{\partial\mathcal{D}}
\newcommand{\clk}{\sigma}
\newcommand{\diin}{\Theta}
\newcommand{\kla}{I^*(H^2)}
\newcommand{\ksm}{I_*(H^2)}

\numberwithin{equation}{section}

\begin{document}

\date{}

\author{Aleksei B.\ Aleksandrov}
\address{St.~Petersburg Department
of Steklov Mathematical Institute, Fontanka 27, St.~Petersburg 191023, Russia, and
St.~Petersburg State University, 7-9 Universitetskaya Emb., St.~Petersburg 199034, Russia}
\email{alex@pdmi.ras.ru}

\author{Evgueni Doubtsov}
\address{Department of Mathematics and Computer Science,
St.~Petersburg State University,
Line 14th (Vasilyevsky Island), 29, St.~Petersburg 199178,
Russia}
\email{dubtsov@pdmi.ras.ru}

\title[Dominant sets for model spaces]{Dominant sets for model spaces\\ in several variables}

\begin{abstract}
Let $I$ be an inner function in $\mathcal{D} = B_{n_1}\times B_{n_2}\cdots \times B_{n_k}$,
where $B_n$ denotes the open unit ball of $\mathbb{C}^n$, $n\ge 1$.
We construct dominant sets for the space $H^2 \ominus I H^2$,
where $H^2 = H^2(\mathcal{D})$ denotes the standard Hardy space.
\end{abstract}

\keywords{Dominant sets, Hardy spaces, large and small model spaces.}

\thanks{The research on Sections~1 and 3 was supported by Russian Science Foundation (grant No.\ 19-11-00058);
the research on Sections~2 and 4 was supported by Russian Science Foundation (grant No.~23-11-00153).}

\maketitle

\section{Introduction}\label{s_int}

Let $\bd$ denote the open unit ball of $\mathbb{C}^n$, $n\ge 1$,
and let $S_n = \partial\bd$ denote the unit sphere.
For the unit disk $B_1$ and the unit circle $\sph_1$, we also use
the symbols $\Dbb$ and $\Tbb$, respectively.

For $k\in \Nbb$ and $n_j\in \Nbb$, $j=1,2,\dots, k$, put
\[
\dom = \dom[n_1, n_2, \dots, n_k] =
B_{n_1}\times B_{n_2}\cdots \times B_{n_k}  \subset \mathbb{C}^{n_1+ n_2 +\dots + n_k}.
\]
Model examples of $\dom$ are the unit ball $\bd$ and the polydisk $\mathbb{D}^k$.

Let $C(z, \zeta) = C_\dom(z, \zeta)$ denote the Cauchy kernel for $\dom$.
Recall that
\[
  C_{\bd}(z, \zeta) = \frac{1}{(1- \langle z, \zeta \rangle)^{n}}, \quad z\in\bd,\ \zeta\in S_n.
\]
Let $\ddo$ denote the set
$\sph_{n_1}\times \sph_{n_2} \cdots \times \sph_{n_k}$,
which is the Shilov boundary of $\dom$ and is also called the distinguished boundary.
Then
\[
\begin{split}
  C_{\dom}(z, \zeta) =\prod_{j=1}^{k} \frac{1}{(1- \langle z_j, \zeta_j \rangle)^{n_j}},
  \quad z&=(z_1, z_2, \dots, z_k)\in\dom,\\ \zeta&=(\zeta_1, \zeta_2, \dots, \zeta_k)\in\ddo,
\end{split}
\]
where $z_j = (z_{j, 1}, z_{j,2}, \dots, z_{j, n_j}) \in B_{n_j}$ and
$\zeta_j = (\zeta_{j, 1}, \zeta_{j,2}, \dots, \zeta_{j, n_j}) \in \sph_{n_j}$.

The corresponding Poisson type kernel is given by the formula
\[
P(z, \zeta) = \frac{C(z, \zeta) C(\zeta, z)}{C(z, z)},
\quad z\in\dom,\ \zeta\in\ddo.
\]
For $\dom=B_n$, $P(\cdot, \cdot)$ is often called the M\"obius invariant Poisson kernel;
see monographs \cite{Ru80, St94} for further details.

\subsection{Inner functions}
Let $\Sigma$ denote the normalized Lebesgue measure on the distinguished boundary $\ddo$.

\begin{df}\label{d_inner}
A holomorphic function $I:\dom \to \Dbb$ is called \textsl{inner} if
$|I(\zeta)|=1$ for $\Sigma$-almost all points $\zeta\in\ddo$.
\end{df}

In the above definition,
$I(\zeta)$ stands, as usual, for
$\lim_{r\to 1-} I(r\zeta)$.
It is well known that the corresponding limit exists for $\Sigma$-almost all points $\zeta\in\ddo$.
Also, observe that by Definition~\ref{d_inner}, unimodular constants are not inner functions.

\subsection{Clark measures}
Let $M(\ddo)$ denote the space of all complex Borel measures on $\ddo$.
For $\mu\in M(\ddo)$ the Poisson integral $P[\mu]$ is defined by the formula
\[
P[\mu](z) = \int_{\ddo} P(z, \zeta)\, d\mu(\zeta), \quad z\in \dom.
\]

Given an $\alpha\in\Tbb$ and an inner function $I : \dom\to \Dbb$, the expression
\[
\frac{1-|I(z)|^2}{|\alpha-I(z)|^2}= \mathrm{Re} \left(\frac{\alpha+ I(z)}{\alpha- I(z)} \right), \quad z\in \dom,
\]
is positive and pluriharmonic.
Thus, there exists a unique positive measure
$\clk_\alpha= \clk_\alpha[I] \in M(\ddo)$
such that
\[
P[\clk_\alpha](z) = \mathrm{Re} \left(\frac{\alpha+ I(z)}{\alpha- I(z)} \right), \quad z\in \dom.
\]
Since $I$ is an inner function, we have
\[
P[\clk_\alpha](\zeta) =\frac{1-|I(\zeta)|^2}{|\alpha-I(\zeta)|^2}=0 \quad \Sigma\textrm{-a.e.},
\]
hence, $\clk_\alpha = \clk_\alpha[I]$ is a singular measure.
Here and in what follows, this means that
$\clk_\alpha$ is singular with respect to Lebesgue measure $\Sigma$.

\subsection{Clark measures and model spaces}
Let $\hol(\dom)$ denote the space of all holomorphic functions in $\dom$.
For $0<p<\infty$, the standard Hardy space $H^p=H^p(\dom)$
consists of $f\in \hol(\dom)$ such that
\[
\|f\|_{H^p}^p = \sup_{0<r<1} \int_{\ddo} |f(r\zeta)|^p\, d\Sigma(\zeta) < \infty.
\]
As usual, the Hardy space $H^p(\dom)$, $p>0$, is identified with the space
$H^p(\ddo)$ of the corresponding boundary values.

For an inner function $\diin$ on $\Dbb$, the classical
model space
$K_\diin$ is defined by the equality
\[
K_\diin = H^2(\Dbb)\ominus \diin H^2(\Dbb),
\]
see, e.g., monograph \cite{NN02}.
Clark \cite{Cl72} established useful connections between Clark measures and model spaces.
In particular, he introduced and studied a family of canonical unitary operators
$U_\alpha : K_\diin \to L^2(\clk_\alpha)$, $\alpha\in\Tbb$.
Further development of the Clark theory in the unit disk and its applications are described, in particular, in \cite{PS06, Sa07}.

Given an inner function $I$ in $\dom$, consider the following natural analogs of the space $K_\diin$:
\begin{align*}
  \kla &= H^2 \ominus I H^2,\\
  \ksm &= \{f\in H^2:\, I\overline{f} \in H^2_0\},
\end{align*}
where $H^2_0 = \{f\in H^2: f(0)=0\}$.
It is clear that
\[
\ksm \subset \kla,
\]
therefore, $\kla$ is called a large model space and $\ksm$ is called a small model space.
Certain basic results about the spaces $\kla$ and $\ksm$ were obtained in the authors' paper \cite{AD20}.
For an inner function $\diin$ in $\Dbb$, we have
 $\diin^* (H^2(\Dbb)) = \diin_* (H^2(\Dbb)) = K_\diin$.

\subsection{Dominant sets for model spaces}\label{ss_dom}
If $\Theta$ is an inner function in the unit disk, then the concept of a dominant set
for the model space $K_\diin$
was introduced in \cite{BFGHR13}.
Such terminology is motivated by the notion of a dominant sequence for the classical space
 $H^\infty$ (see \cite{BSZ60}).
In the present paper,
this terminology is used for analogs of model spaces in $\dom$.

\begin{df}
A Lebesgue measurable set $E\subset \ddo$ is said to be \textsl{dominant}
 for the large model space
$\kla$ if $\Sigma(E) < 1$ and
\begin{equation}\label{e_dom_df}
\|f\|_{H^2}^2
\le C\int_E
|f|^2\, d\Sigma
\end{equation}
for all $f\in \kla$.
\end{df}
On the one hand, the condition $\Sigma(E)<1$ excludes the corresponding trivial examples of dominant sets.
On the other hand, property~\eqref{e_dom_df} implies the inequality $\Sigma(E)>0$.

\subsection{Existence of dominant sets}
For inner functions $\diin$ defined on $\Dbb$ and having additional properties,
a number of examples of dominant sets were constructed in \cite{BFGHR13}.
Thus, it is natural to ask whether there are dominant sets for an arbitrary model space.
Kapustin (see \cite[Theorem~5.14]{BFGHR13})
obtained a positive answer to this question
in the case of the unit disk.
The main result of the present paper is the following existence theorem related
to the large model spaces in the multidimensional domain $\dom$.

\begin{thm}\label{t_dom_exists}
Let $I$ be an inner function in $\dom$.
Then there exists a dominant set for the large model space $\kla$.
\end{thm}

\subsection*{Organization of the paper}
Auxiliary results, in particular, a result on
the disintegration of Lebesgue measure in terms of Clark measures are collected in Section~\ref{s_aux}.
Theorem~\ref{t_dom_exists} is proved in Section~\ref{s_reDom}.
The existence of radial limits almost everywhere with respect to Clark measures is discussed
in the final Section~\ref{s_radial}.

\section{Auxiliary results}\label{s_aux}

\subsection{Supports of singular Clark measures}
Let $\alpha\in\Tbb$ and $I$ be an inner function in $\dom$.
Since $\clk_\alpha$ is a singular measure, well known properties
of Poisson integrals guarantee the following property:
\begin{equation}\label{e_poiss_al}
I(\zeta) =\alpha\quad \textrm{for\ } \clk_\alpha \textrm{-almost all points}\ \zeta\in \ddo.
\end{equation}
In particular, $\clk_\alpha$ and $\clk_\beta$ are mutually singular for $\alpha\neq\beta$.

\subsection{Disintegration of Lebesgue measure}
Let $\meal$ denote the normalized Le\-besgue measure on the unit circle $\Tbb$.
For $\dom = \Dbb$, the following result was obtained in \cite{A87}.

\begin{pr}\label{t_desint}
Let $I: \dom\to\Dbb$ be an inner function, $\alpha\in \Tbb$ and let
$\clk_\alpha = \clk_\alpha[I]$ be the corresponding Clark measure.
Then
\begin{equation}\label{e_desint}
\int_{\Tbb}
\int_{\ddo}  g\, d\clk_\alpha\, d\meal(\alpha) = \int_{\ddo} g\, d\Sigma
\end{equation}
for all $g\in L^1(\ddo)$,
where equality \eqref{e_desint} is understood in the following weak sense:
for $m$-almost all
$\alpha\in\Tbb$, the function $g$ is defined $\clk_\alpha$-a.e.\
and is summable with respect to $\clk_\alpha$.
\end{pr}
\begin{proof}
If $g\in C(\ddo)$, then it suffices
to repeat the argument used in \cite[Theorem~3.3]{AD20} for $\dom=\bd$.
The special case $g\in C(\ddo)$ implies the general one, since
equality \eqref{e_desint} is understood in the above weak sense.
\end{proof}

Proposition \ref{t_desint} easily implies the following assertion.

\begin{cor}\label{c_sequence}
Let $\{f_n\}_{n=1}^\infty$ be a functional sequence
converging to zero in $L^2(\Sigma)$. Then there exists a subsequence
 $\{f_{n_k}\}_{k=1}^\infty$ such that
$\lim_{k\to\infty}\|f_{n_k}\|_{L^2(\clk_\alpha)}=0$ for $m$-almost all $\alpha\in\Tbb$.
\end{cor}

\subsection{Cauchy integrals and Clark measures}\label{s_aux_clk}
The following technical assertion is proved in \cite[Proposition~2.2]{AD23}.

\begin{pr}\label{p_cauchy_dbl}
Let $I : \dom\to\Dbb$ be an inner function, $\alpha\in \Tbb$ and let
$\clk_\alpha = \clk_\alpha[I]$ be the corresponding Clark measure.
Then
  \[
  \int_{\ddo} C(z, \zeta) C(\zeta, w)\, d\clk_\alpha(\zeta) =
  \frac{1- I(z)\overline{I(w)}}{(1-\overline{\alpha}{I(z)})(1-\alpha\overline{I(w)})} C(z,w)
  \]
for all $\alpha\in\Tbb$, $z, w \in\dom$.
\end{pr}

\subsection{Reproducing kernels for the spaces $\kla$}
Let $I$ be an inner function in $\dom$. Then
\[
k(z, \zeta) = k(I; z, \zeta) \overset{\textrm{def}}{=}
(1-I(z)\overline{I(\zeta)}) C(z, \zeta)
\]
is the reproducing kernel for the large model space $\kla$.
Indeed, $C(z, \zeta)$ is the reproducing kernel for $H^2(\dom)$, hence,
$I(z) C(z, \zeta) \overline{I(\zeta)}$ is the reproducing kernel for $I H^2(\dom)$.
Therefore, the difference
$C(z,\zeta) - I(z) C(z, \zeta) \overline{I(\zeta)}$
is the reproducing kernel for the space $H^2(\dom) \ominus I H^2(\dom)$.

\section{Dominant sets for model spaces}\label{s_reDom}

\subsection{Preliminary results}
\begin{pr}\label{p_Scalar_Prod}
Let $I: \dom\to\Dbb$ be an inner function, $\alpha\in \Tbb$ and let
$\clk_\alpha = \clk_\alpha[I]$ be the corresponding Clark measure.
Then, for any $f, g\in\kla$, the equality
\begin{equation}\label{e_Scalar_Prod}
(f, g)_{H^2}=
\int_{\ddo}  f{\overline g}\, d\clk_\alpha
\end{equation}
holds for $m$-almost all $\alpha\in\Tbb$.
\end{pr}
\begin{proof}
Put $k_z(\zeta) = k(z, \zeta)$.
Successively applying explicit formulas for the kernels
$k_z(\zeta)$ and $k_w(\zeta)$, $z,w\in\dom$, property \eqref{e_poiss_al}
and Proposition~\ref{p_cauchy_dbl}, we obtain
\[
\begin{split}
   \int_{\ddo} k_z(\zeta) \overline{k_w (\zeta)}
&\, d\clk_\alpha[I](\zeta) \\
&= \int_{\ddo} \left(1-\overline{I(z)}I(\zeta)\right) C(\zeta, z)
    \left(1-I(w)\overline{I(\zeta)}\right) C(w, \zeta) \, d\clk_\alpha[I](\zeta) \\
&= (1-\alpha\overline{I(z)})(1-\overline{\alpha}I(w)) \int_{\ddo} C(\zeta, z) C(w, \zeta)\, d\clk_\alpha[I](\zeta) \\
&= \left(1-\overline{I(z)}I(w)\right) C(w,z)\\
&= k(w,z) = (k_z, k_w)_{H^2}.
\end{split}
\]
Thus, for all $\alpha\in\Tbb$, equality~\eqref{e_Scalar_Prod}
holds for all finite linear combinations of $k_z$, $z\in\dom$.
In particular, the equality
\begin{equation}\label{e_Norm}
\|f\|_{H^2}^2=
\int_{\ddo}  |f|^2\, d\clk_\alpha
\end{equation}
holds for $m$-almost all $\alpha\in\Tbb$ for any function $f$ in the
linear span of the family $\{k_z\}_{z\in\dom}$.
Applying Corollary~\ref{c_sequence}, we extend this equality to all functions
$f$ in the space $\kla$.
To deduce equality \eqref{e_Scalar_Prod} from \eqref{e_Norm}, it suffices
to apply the standard formula expressing the scalar product in terms of the corresponding norms.
\end{proof}

Before formulating the next assertion, observe that the composition
$\Phi\circ I$ is well defined for any function
 $\Phi\in L^1(\Tbb) = L^1(\Tbb, \meal)$.
Indeed, if $I(0)=0$, then $\Sigma(I^{-1}(Q))=m(Q)$ for any measurable set $Q\subset\Tbb$,
where $I^{-1}(Q) = \{\zeta\in\ddo:\, I(\zeta) \in  Q\}$.
If $I$ is an arbitrary inner function, then we have
$\psi\circ I(0)=0$, where
\[
\psi(z)=\frac{I(0)-z}{1-\overline{I(0)}z}, \quad z\in\Dbb.
\]
Therefore, the properties $\Sigma(I^{-1}(Q))=0$ and $m(Q)=0$ are equivalent;
hence, the composition $\Phi\circ I$ is defined almost everywhere and is measurable.

\begin{cor}\label{c_comp}
Let $\Phi\in L^1(\Tbb)$. Then
\[
\int_{\ddo} (\Phi\circ I) f\overline g\,d\Sigma=\left(\int_{\Tbb}\Phi\, dm\right) (f,g)_{H^2}
\]
for all $f,g\in\kla$.
\end{cor}
\begin{proof}
Applying Proposition~\ref{p_Scalar_Prod}, property~\eqref{e_poiss_al} and Proposition~\ref{t_desint},
we obtain
\[
\begin{split}
\left(\int_{\Tbb}\Phi\, dm\right)\!(f,g)_{H^2}
&= \int_{\Tbb} \Phi(\alpha) (f,g)_{H^2}\, d\meal(\alpha) \\
&=\int_{\Tbb}\int_{\ddo} \Phi(\alpha) f(\zeta){\overline g}(\zeta)\, d\clk_\alpha(\zeta)\, d\meal(\alpha)  \\
&=\int_{\Tbb}\int_{\ddo} (\Phi\circ I) f{\overline g}\, d\clk_\alpha\, d\meal(\alpha) \\
&=\int_{\ddo} (\Phi\circ I) f\overline g\,d\Sigma,
\end{split}
\]
as required.
\end{proof}

\subsection{Proof of Theorem~\ref{t_dom_exists}}
Fix a measurable set $Q\subset \Tbb$ such that $0<\meal(Q)<1$.
Consider the set $I^{-1}(Q)$.
Applying Corollary~\ref{c_comp} with $\Phi=\chi_Q$, we obtain
\[
m(Q)\|f\|_{H^2}^2=\int_{I^{-1}(Q)}|f|^2\,d\Sigma, \quad f\in\kla,
\]
in particular, estimate \eqref{e_dom_df} holds for $E = I^{-1}(Q)$.
Next, recall that $\Sigma(I^{-1}(Q))=m(Q)$ provided that $I(0)=0$.
Thus, for an arbitrary inner function $I$, the condition $\meal(Q)<1$
guarantees that $\Sigma(I^{-1}(Q))<1$.
Therefore, $I^{-1}(Q)$ is a dominant set, and
the proof of Theorem~\ref{t_dom_exists} is finished.

\section{Radial behavior of functions from model spaces}\label{s_radial}

Let $I:\dom\to \Dbb$ be an inner function.
As indicated in Proposition~\ref{t_desint}, if $g\in L^1(\ddo)$, then
for $m$-almost all $\alpha\in\Tbb$, the function $g$ is defined
$\clk_\alpha$-a.e.\ and is summable with respect to $\clk_\alpha = \clk_\alpha[I]$.
Given an $f\in\kla$, one can draw similar conclusions on the existence of radial limits
$f(\zeta) = \lim_{r\to 1-} f(r\zeta)$.
Indeed, the radial limit $f(\zeta)$ is defined for $\Sigma$-almost all $\zeta\in\ddo$.
Hence, properties of the family $\{\clk_\alpha\}_{\alpha\in\Tbb}$
guarantee that for $\meal$-almost all $\alpha\in\Tbb$, the value $f(\zeta)$
is defined for $\clk_\alpha$-almost all points $\zeta\in\ddo$.
Therefore, it is natural to ask whether the corresponding statement holds for \textit{all} $\alpha\in\Tbb$.
Poltoratski \cite{Po93} obtained a positive answer to this question for $\dom=\Dbb$.

\subsection{Poltoratski's theorem for the space $\ksm$}
In this section, we show that
consideration of slice functions allows one to extend
Poltoratski's theorem to the functions from the small model space $\ksm$.

Let $\alpha\in\Tbb$ and $\clk_\alpha = \clk_\alpha[I]$ be a Clark measure.
Put $u=P[\clk_\alpha]$. For $\xi\in\ddo$, consider the slice function
$u_\xi(z) := u(z\xi)$, $z\in\Dbb$.
Since $u_\xi$ is a positive harmonic function on $\Dbb$,
it is the one-dimensional Poisson integral of a positive measure defined on $\Tbb$.
We use the notation $(\clk_\alpha)_\xi$ for the corresponding slice-measure.

The following formula for integration on slices is well known:
\begin{equation}\label{e_slices}
 \int_{\ddo} g\, d\Sigma = \int_{\ddo} \int_{\Tbb} g(\lambda\zeta)\, d\meal(\lambda)\, d\Sigma(\zeta), \quad g\in C(\ddo).
\end{equation}
The above formula guarantees (see, e.g., \cite[Proposition~2.1]{AD20} in the case, where $\dom=\bd$) that
\begin{equation}\label{e_clk_decomp}
\int_{\ddo} g\, d\clk_\alpha = \int_{\ddo} \int_{\Tbb} g(\lambda\xi)\, d(\clk_\alpha)_\xi (\lambda)\, d\Sigma(\xi)
\end{equation}
for all $g\in C(\ddo)$.
Moreover, standard arguments show that equality \eqref{e_clk_decomp} extends to all
bounded Borel functions $g$ defined on $\ddo$.

\begin{pr}\label{p_pltr}
Let $I$ be an inner function in $\dom$, $f\in\ksm$ and $\alpha\in\Tbb$.
Then the radial limit $\lim_{r\to 1-} f(r\zeta)$ exists for
$\clk_\alpha$-almost all points $\zeta\in\ddo$.
\end{pr}
\begin{proof}
Let $f\in\ksm$.
By \eqref{e_slices}, the slice functions $I_\xi$ and
$f_\xi$ have the following properties for $\Sigma$-almost all $\xi\in\ddo$:
\begin{itemize}
  \item $I_\xi$ is an inner function in $\Dbb$,
  \item $f_\xi\in (I_\xi)_*(H^2(\Dbb)) \subset H^2(\Dbb)$.
\end{itemize}
Let $\alpha\in\Tbb$. Assume that the above properties hold for $\xi\in\ddo$.
Since $(I_\xi)_*(H^2(\Dbb))$ is a classical model space and $f_\xi\in (I_\xi)_*(H^2(\Dbb))$,
the radial limit
$\lim_{r\to 1-} f_\xi(r\lambda)$ exists for $\clk_\alpha[I_\xi]$-almost all $\lambda\in\Tbb$
by Poltoratski's theorem \cite{Po93}.
Now, observe that the Clark measure $\clk_\alpha[I_\xi]$
coincides with the slice-measure $(\clk_\alpha[I])_\xi$.
Hence $(\clk_\alpha[I])_\xi(E_\xi) =0$, where $E$ denotes the set of those points $\zeta\in\ddo$,
for which the limit $\lim_{r\to 1-} f(r\zeta)$ does not exist, and $E_\xi = \{\lambda\in\Tbb: \lambda\xi\in E\}$.
Since the stated property holds for $\Sigma$-almost all points $\xi\in\ddo$,
formula \eqref{e_clk_decomp}
with $g=\chi_E$
guarantees that $\clk_\alpha(E)=0$, as required.
\end{proof}

\bibliographystyle{amsplain}
\providecommand{\bysame}{\leavevmode\hbox to3em{\hrulefill}\thinspace}
\providecommand{\MR}{\relax\ifhmode\unskip\space\fi MR }
\providecommand{\MRhref}[2]{%
  \href{http://www.ams.org/mathscinet-getitem?mr=#1}{#2}
}
\providecommand{\href}[2]{#2}

\end{document}